\documentclass[11pt,a4paper]{scrartcl}
\usepackage{a4wide}
\usepackage{tikz}
\usetikzlibrary{matrix,arrows}
\usepackage{tikz-cd}
\usepackage{amsmath, amsthm, amscd, amsfonts, amssymb, graphicx, color}
\usepackage[bookmarksnumbered, plainpages]{hyperref}

\newcommand{\Z}{\ensuremath{\mathbb{Z}}}    % Set of integers - Z
\newcommand{\C}{\ensuremath{\mathbb{C}}}    % Set of complex numbers - C
\newcommand{\comL}[1]{[#1]}  % Commutator in the Lie ring

\let\xto\xrightarrow

\DeclareMathOperator{\Inf}{Inf}
\DeclareMathOperator{\Tra}{Tra}
\DeclareMathOperator{\Res}{Res}
\DeclareMathOperator{\Hom}{Hom}
\DeclareMathOperator{\id}{id}
\DeclareMathOperator{\im}{Im}

\newcommand{\Defn}[1]{{\emph{#1}}}
\newtheorem{thm}{Theorem}[section]
\newtheorem{cor}[thm]{Corollary}
\newtheorem{lem}[thm]{Lemma}
\newtheorem{prop}[thm]{Proposition}
\newtheorem{defn}[thm]{Definition}

\numberwithin{equation}{section}

\begin{document}

\title{Second cohomology of Lie rings and the Schur multiplier}
\author{Max Horn}
\author{Seiran Zandi}

\maketitle

\begin{abstract}
We exhibit an explicit construction for the second cohomology group 
$H^2(L, A)$ for a Lie ring $L$ and a trivial $L$-module $A$.
We show how the elements of $H^2(L, A)$ correspond one-to-one to the 
equivalence classes of central extensions of $L$ by $A$, where $A$
now is considered as an abelian Lie ring. For a finite Lie
ring $L$ we also show that $H^2(L, \C^*) \cong M(L)$, where $M(L)$ denotes the 
Schur multiplier of $L$. These results match precisely the analogue
situation in group theory.

\medskip

\noindent
\textit{Keywords:} Lie rings, Schur multiplier of Lie rings, central extension, second cohomology group of Lie rings.
\\
\textit{MSC(2010):} Primary: 17B56; Secondary: 20F40, 20D15.
\end{abstract}

\section{Introduction} \label{sec:intro}
%  ___         _   _          
% / __| ___ __| |_(_)___ _ _  
% \__ \/ -_) _|  _| / _ \ ' \ 
% |___/\___\__|\__|_\___/_||_|

Let $L$ be a Lie ring and $A$ a trivial $L$-module. Our first aim in this
paper is to give an explicit description of the cohomology group $H^2(L,A)$
and to show how its elements correspond one-to-one
to the equivalence classes of central extensions of the Lie algebra $L$
with the module $A$, where we regard $A$ as abelian Lie ring. We then prove the following first main result.

\begin{thm}[see Theorem~\ref{thm:exact-sequence}] \label{mainthm:exact-sequence}
  Let $H$ be a central ideal of $L$.
  Then there is an exact sequence
  \[ \pushQED{\qed}
  0 \to \Hom(L/H,A)\to\Hom(L,A)\to\Hom(H,A)\to H^2(L/H,A)\to H^2(L,A).
  \qedhere
  \]
\end{thm}

The Schur multiplier $M(L)$ of a Lie ring $L$ is defined as $M(L)=H_2(L,\Z)$.
This has been investigated in \cite{Bak:2007ei} and \cite{Eick/Horn/Zandi}.
Our second main result here relates the second cohomology group of a Lie ring with its Schur
multiplier.

\begin{thm}[see Theorem~\ref{thm:H^2-from-FR}] \label{mainthm:H^2-from-FR}
  Let $L=F/R$ be a finite Lie ring, where $F$ is a free Lie ring
  and $R$ an ideal of $F$. Consider $\C^*$ as a trivial $L$-module.
  Then $H^2(L,\C^*)\cong (R \cap F^2) / [F,R]$.
  \qed
\end{thm}

Combining this with \cite[Theorem 3.3.1]{Bak:2007ei} we obtain the following Corollary.

\begin{cor} \label{cor:ML=H^2}
 Let $L$ be a finite Lie ring. Then $M(L) = H_2(L,\Z) \cong H^2(L,\C^*)$.
 \qed
\end{cor}

Our results have analogues in group theory and in the theory of Lie algebras.
But note that while Lie algebras are also Lie rings,
it is not true that all corresponding results on Lie algebras follow as a consequence of
our work. Indeed, the categories of Lie rings and $k$-Lie algebras (where $k$
is a field) are quite different. For example, they have different free objects,
so the notions of presentations do not coincide directly. Also, while every
central extension of Lie algebras is a central extension of Lie rings,
the converse is not true. For example, let $p$ be a prime and consider
the Lie algebra $L=\Z/p\Z$. Then the only central extension as a Lie algebra
of $L$ by itself is $L\times L$. But as a Lie ring, there is also the
central extension $\Z/p^2\Z$, which is not a vector space and hence not a Lie
algebra. Therefore, our proof of the correspondence between central extensions
and $H^2(L,A)$ does not imply the similar result for Lie algebras.

\medskip

We now briefly discuss the analogues of our results for groups and Lie algebras.

The five-term exact
sequence from Theorem~\ref{mainthm:exact-sequence} is analogue to a
special case of the well-known inflation-restriction exact sequence,
where the Lie ring $L$ and its central ideal $H$ are replaced by a group
with a central subgroup, respectively by a Lie algebra with a central
ideal.

In turn, Theorem~\ref{mainthm:H^2-from-FR} corresponds to a
well-known result for finite groups proved by Schur in \cite{Schur:1904dk,Schur:1907fg},
which is also known for Lie algebras, see e.g.\ \cite{Batten/Stitzinger:1996db}.

Finally, note that \cite[Theorem 3.3.1]{Bak:2007ei} we used above to establish
Corollary~\ref{cor:ML=H^2} is a Lie rings analogue
of the Hopf formula for groups, see \cite{Hopf}. Again there are
similar results for Lie algebras, see e.g.\ \cite[Section 7.5]{Weibel:1994}.
Thus if $X$ is a group, Lie algebra or Lie ring, and finite, then $H_2(X,\Z) \cong H^2(X,\C^*)$.

%------------------------------------------------------------------------------------%

\section{Preliminaries} \label{sec:prelim}
%  ___         _   _          
% / __| ___ __| |_(_)___ _ _  
% \__ \/ -_) _|  _| / _ \ ' \ 
% |___/\___\__|\__|_\___/_||_|

We collect some definitions and facts that we use throughout this paper.
A Lie ring $L$ is an additive abelian group with a (not necessarily
associative) multiplication denoted by $\comL{\cdot,\cdot}$ satisfying
the following properties:
\begin{enumerate}
  \item
  $[x,x]=0$ for all $x\in L$ (anti-commutativity),
  \item
    $[x,[y,z]]+[z,[x,y]]+[y,[z,x]]=0$ for all $x,y,z\in L$
      (Jacobi identity),
  \item
  $[x+y,z]=[x,z]+[y,z]$ and $[x,y+z]=[x,y]+[x,z]$ for all
                   $x,y,z\in L$ (bilinearity).
\end{enumerate}
The product $\comL{x,y}$ is called the \Defn{commutator} of $x$ and $y$. 
Subrings, ideals and homomorphisms of Lie rings are defined as usual, and we write
$K\leq L$ and $I\unlhd L$ if $K$ is a subring and $I$ an ideal of $L$.
The \Defn{center} of $L$ is an ideal defined as
\[
Z(L):= \{ x\in L|\:[x,y]=0\:\text{for all } y\in L\}.
\]
If $L=Z(L)$ then $L$ is called \Defn{abelian}.

Given a Lie ring $L$ and two Lie subrings $U$ and $V$, we define $[U,V]$ as
the Lie subring of $L$ generated by all commutators $[u,v]$ with $u \in U$
and $v \in V$. We also write $U^2$ instead of $[U,U]$.

Free Lie rings and (finite) presentations of Lie rings
are again defined as usual, see \cite[§5.3]{khukhro} for details.

\section{Second cohomology of Lie rings} \label{sec:2nd-cohom}
%  ___         _   _          
% / __| ___ __| |_(_)___ _ _  
% \__ \/ -_) _|  _| / _ \ ' \ 
% |___/\___\__|\__|_\___/_||_|

In this section we give an explicit construction of
the second cohomology group of Lie rings.
Throughout this section $L$ is a Lie ring and $A$ is a trivial $L$-module.

\begin{defn}\label{defn:lie-ring-cocycles}
Let $Z^2(L,A)$ be the set of all pairs of  functions $(f,g)$,
where $f,g : L\times L \to A$ satisfy the following conditions for all $x,y,z\in L$:

\begin{enumerate}
\item \label{defn:cocycle-bilinear-f}
      $f(x+y , z)=f(x,z)+f(y,z)+g(\comL{x,z},\comL{y,z})$ and \\
      $f(x,y+z)=f(x,y)+f(x,z)+g(\comL{x,y},\comL{x,z})$.
\item \label{defn:cocycle-anti-comm-f} $f(x,x)=0$.
\item \label{defn:cocycle-jacobi}
      $f(x,\comL{y,z})+f(z,\comL{x,y})+f(y,\comL{z,x})=
       -g( \comL{x,\comL{y,z}},\comL{z,\comL{x,y}})
       -g(-\comL{y,\comL{z,x}},\comL{y,\comL{z,x}})$.
\item \label{defn:cocycle-associative-g}
      $g(x+y,z)+g(x,y)=g(x,y+z)+g(y,z)$.
\item \label{defn:cocycle-comm-g}
      $g(x,y)=g(y,x)$.
\end{enumerate}
We shall refer to the elements of $Z^2(L,A)$ as \Defn{cocycles} of $L$ in $A$.
\end{defn}
Given two cocycles $(f,g), (f',g')\in Z^2(L,A)$, we define their sum by
\[ (f,g)+(f',g'):=(f+f',g+g'), \]
where $f+f'$ and $g+g'$ denote pointwise addition of the two functions.
It is now straightforward to verify that the set $Z^2(L,A)$ endowed with
this operation is an abelian group.

\begin{defn}
Let $t:L\to A$ be a map such that $t(0)=0$. We define
  \[
  \begin{array}{lll}
  f:& L\times L \to A ,& (x,y) \mapsto -t([x,y]), \\
  g:& L\times L \to A ,& (x,y) \mapsto  t(x)+t(y)-t(x+y);
  \end{array}
  \]
  and refer to the pair $(f,g)$ as a \Defn{coboundary} of $L$ in $A$, and denote the
  set of all coboundaries by $B^2(L,A)$.
\end{defn}
One can readily verify that $B^2(L,A)$ is a a subset and in fact a subgroup of $Z^2(L,A)$.
 The factor group $H^2(L,A):={Z^2(L,A)}/{B^2(L,A)}$ is called the \Defn{second cohomology group of $L$ with coefficients in $A$},
  the elements $H^2(L,A)$ are called \Defn{cohomology classes}. Any two cocycles contained in the same
 cohomology class are said to be \Defn{cohomologous}.
Given $(f,g)\in Z^2(L,A)$ we write $\overline{(f,g)}$ to denote the cohomology class containing $(f,g)$.

\section{From central extensions to cohomology and back} \label{sec:centext-to-cohom}
%  ___         _   _          
% / __| ___ __| |_(_)___ _ _  
% \__ \/ -_) _|  _| / _ \ ' \ 
% |___/\___\__|\__|_\___/_||_|

Let $L$ be a Lie ring and $A$ an abelian Lie ring.
In this section we study the connection between equivalence classes of
central extensions of $L$ by $A$ on the one hand, and the
cohomology group $H^2(L,A)$ on the other hand, where $A$ is regarded
as a trivial $L$-module.
\medskip

A \Defn{central extension} of $L$ by $A$ is a short exact sequence
\[\begin{CD}
 E:\quad @. 0 @>>> A @>{\alpha}>> B @>{\beta}>> L @>>> 0\\
 \end{CD}
\]
of Lie rings such that $\alpha(A)$ is contained in the center of $B$.
A \Defn{section} of $\beta$ is a map $\lambda:L\to B$ such that $\beta\circ \lambda=\id_L$.
If a section of $\beta$ exists that is a homomorphism,
we say that the extension $E$ is \Defn{split}.

Two central extensions $E$ and $E'$  of $L$ by $A$ are said to be \Defn{equivalent}
if there exists a Lie ring homomorphism $\gamma$ which makes the
following diagram commutative:
\[
\begin{CD}
  E:\quad  @. 0 @>>> A  @>{\alpha}>>  B            @>{\beta}>>  L  @>>> 0 \\
  @.@.     @|               @VV{\gamma}V              @|      @.\\
  E':\quad @. 0 @>>> A  @>{\alpha'}>> B'           @>{\beta'}>> L  @>>> 0 \\
\end{CD}
\]\\
By the well-known Five Lemma (whose standard proof in the category
of groups via diagram chasing transfers almost verbatim to the category of Lie rings)
it follows that $\gamma$ is an isomorphism.

Observe that any central extension is equivalent to one in which $\alpha$ is simply the inclusion map. So we will from now on always assume that:
\[\begin{CD}
  E:\quad  @. 0 @>>> A  @>{i}>>  B  @>{\beta}>>  L  @>>> 0
\end{CD}
\]
where $i$ denotes the inclusion map.
Let $\lambda$ denote a section of $\beta$ such that $\lambda(0)=0$.
Since $\beta([\lambda(x),\lambda(y)])=\beta(\lambda[x,y])$, we find that
$[\lambda(x),\lambda(y)]$ and $\lambda[x,y]$ belong to the same coset of $A$. Similarly, we can argue
that $\lambda(x)+\lambda(y)$ and $\lambda(x+y)$ are in the same coset of $A$.
Hence we may define the following maps:
  \[
  \begin{array}{lll}
  f:& L\times L \to A ,& (x,y) \mapsto [\lambda(x),\lambda(y)] - \lambda[x,y], \\
  g:& L\times L \to A ,& (x,y) \mapsto \lambda(x)+\lambda(y) - \lambda(x+y)
  .
  \end{array}
  \]
Intuitively, these two maps measure how far away $\lambda$ is from being a Lie ring homomorphism.

\begin{lem}
 $(f,g)$ is a cocycle.
\end{lem}
\begin{proof}
It is easy to show that $(f,g)$ has properties $(2)$, $(4)$ and $(5)$.
To verify the property $(1)$, consider $x,y$ and $z\in L$. By using that $\im g\in A\leq Z(L)$, we get
  \begin{align*}
    f(x,y+z)
    &= [\lambda(x),\lambda(y+z)]-\lambda[x,y+z] \\
    &= [\lambda(x),\lambda(y)+\lambda(z)-g(y,z)]-\lambda([x,y]+[x,z]) \\
    &= [\lambda(x),\lambda(y)]+[\lambda(x),\lambda(z)]-\lambda[x,y]-\lambda[x,z] + g([x,y],[x,z]) \\
    &= f(x,y)+f(x,z)+g([x,y],[x,z]).
  \end{align*}
Similarly we obtain $f(x+y,z)=f(x,z)+f(y,z)+g([x,z],[x,y])$. It remains to verify the condition $(3)$.
Since $L$ and $B$ are Lie rings, using the Jacobi identity, we have:
  \begin{align*}
    & f(x,[y,z])+f(z,[x,y])+f(y,[z,x]) \\
    &= [\lambda(x),[\lambda(y),\lambda(z)]]+[\lambda(z),[\lambda(x),\lambda(y)]]+[\lambda(y),[\lambda(z),\lambda(y)]]\\
      &\qquad\qquad-\lambda[x,[y,z]]-\lambda[z,[x,y]]-\lambda[y,[z,x]]\\
    &= -g([x,[y,z]],[z,[x,y]])-g(-[y,[z,x]],[y,[z,x]]).
  \end{align*}
Thus $(f,g)$ is a cocycle as claimed.
\end{proof}
\begin{lem}
 The cohomology class of $(f,g)$ does not depend on the choice of $\lambda$.
\end{lem}
\begin{proof}
Assume that $\mu$ is another section of $\beta$ such that $\mu(0)=0$.
For all $x \in L$, since $\mu(x)$ and $\lambda(x)$ belong to the same coset of $A$, there is
a function $t:L\to A$ such that $\mu(x)=t(x)+\lambda(x)$. If $(f',g')$ denotes the cocycle
corresponding to $\mu$, then
  \begin{align*}
    f'(x,y)
    &= [\mu(x),\mu(y)]-\mu[x,y] \\
    &= [t(x)+\lambda(x),t(y)+\lambda(y)]-(t+\lambda)[x,y]\\
    &= [\lambda(x),\lambda(y)]-t[x,y]-\lambda[x,y]\\
    &= f(x,y)-t[x,y],
  \end{align*}
and of course $g'(x,y)=g(x,y)+t(x)+t(y)-t(x+y)$. Thus $(f,g)$ and $(f',g')$ are cohomologous.
\end{proof}
It follows that every extension $E$ of $L$ by $A$ determines a unique cohomology class $C_E$ in $H^2(L,A)$.
We will now establish that this actually yields a bijection between the equivalence classes of extensions and the
cohomology classes.
We need the following result which is readily proved.
\begin{prop}
Suppose $E$ and $E'$ are central extensions.
Then $E$ and $E'$ are equivalent if and only if $C_E=C_{E'}$ holds.
\end{prop}
Finally, we may connect equivalence classes of extensions with the cohomology classes.

\begin{prop}\label{prop:h^2-vs-centext}
 The assignment $E \mapsto C_E$ determines a bijective correspondence between the equivalence classes of central
extensions of $L$ by $A$, and the second cohomology group $H^2(L,A)$. Furthermore, if $E$ is a split extension,
then $C_E= 0 + B^2(L,A)$.
\end{prop}
\begin{proof}
From the preceding proposition, we know that the map from equivalence
classes of central extensions to cohomology classes is injective.
It remains to show that it is also surjective.

To see this, consider an arbitrary cocycle $(f,g)\in Z^2(L,A)$.
Then we define an addition and a bracket on the set $B=A\times L$ by putting
  \begin{align*}
    (a,x)+(b,y) &:=(a+b+g(x,y),x+y), \\
    [(a,x),(b,y)] &:=(f(x,y),[x,y]),
  \end{align*}
for any $(a,x),(b,y) \in B$.
First we check that $(B,+,\comL{\cdot,\cdot})$ is a Lie ring.
Suppose that $(a,x)$, $(b,y)$ and $(c,z)$ are in $B$.
That $+$ defines an associative and commutative binary operation follows from $(4)$ and $(5)$.
The bracket operation is bilinear on the left:
  \begin{align*}
    [(a,x)+(b,y),(c,z)]
    &= [(a+b+g(x,y),x+y),(c,z)]\\
    &= (f(x+y,z),[x+y,z]) \\
    &=
       (f(x,z)+f(y,z)+g([x,z],[y,z]),[x,z]+[y,z])\\
    &= (f(x,z),[x,z])+(f(y,z),[y,z])\\
    &= [(a,x),(c,z)]+[(b,y),(c,z)]
    .
  \end{align*}
Similarly we can show $[(a,x),(b,y)+(c,z)]=[(a,x),(b,y)]+[(a,x),(c,z)]$.
Hence the bracket operation is bilinear. We next turn our attention to the Jacobi identity.
Since $L$ is a Lie ring, we have
  \begin{align*}
    &\quad
       [(a,x),[(b,y),(c,z)]] +
       [(c,z),[(a,x),(b,y)]] +
       [(b,y),[(c,z),(a,x)]]\\
    &= [(a,x),(f(y,z),[y,z])] +
       [(c,z),(f(x,y),[x,y])] +
       [(b,y),(f(z,x),[z,x])]\\
    &= (f(x,[y,z]),[x,[y,z]]) +
       (f(z,[x,y]),[z,[x,y]]) +
       (f(y,[z,x]),[y,[z,x]])\\
    &= (f(x,[y,z])+f(z,[x,y])+g([x,[y,z]],[z,[x,y]]),[x,[y,z]]+[z,[x,y]])\\
    &\quad +(f(y,[z,x]),[y,[z,x]])\\
    &
      =(-f(y,[z,x])-g(-[y,[z,x]],[y,[z,x]]),-[y,[z,x]])
       +(f(y,[z,x]),[y,[z,x]])\\
    &= (0,0).
  \end{align*}
Finally by the condition $(2)$,
  \[
    [(a,x),(a,x)] =(f(x,x),[x,x]) = (0, 0).
  \]
We conclude that $B$ is a Lie ring. Now  $\{(a,0) \mid a\in A\}$ is an isomorphic image of $A$ contained in
the center of $B$. Identifying $A$ with this image, we obtain a
central extension of $L$ by $A$:
\[
\begin{CD}
  E:\quad  @. 0 @>>> A  @>>>  B   @>{\beta}>>  L  @>>> 0, \\
\end{CD}
\]
where $\beta(a,x)=x$ for all $a\in A$ and $x\in L$. Let $\lambda:L\to B$ be the
section of $B$ defined by $\lambda(x)=(0,x)$. Then by definition of addition and bracket
on $B$, we see that $(f,g)$ is the cocycle
corresponding to $E$. We conclude that the map sending ${E}$ to $C_E$ is surjective,
hence bijective.

Finally, assume that $E$ is a split extension. Then there is a homomorphism
$\alpha:L\to B$ which is a section of $B$. Accordingly the cocycle corresponding to
$\alpha$ is a coboundary, hence $C_E=0 + B^2(L,A)$.
\end{proof}

\section{An exact sequence for Lie rings} \label{sec:hs-seq}
%  ___         _   _          
% / __| ___ __| |_(_)___ _ _  
% \__ \/ -_) _|  _| / _ \ ' \ 
% |___/\___\__|\__|_\___/_||_|

In this section we show that for a central ideal $H$ of $L$,
the following sequence of Lie rings is exact:
\[ 0 \to \Hom(L/H,A)\to\Hom(L,A)\to\Hom(H,A)\to H^2(L/H,A)\to H^2(L,A). \]
We also construct the connecting maps explicitly. Note that this sequence
is a Lie ring analogue for the well-known 5-term short exact sequence for groups. 

\medskip

Let $L\cong F/R$ be a presentation of $L$, that is $F$ is a free Lie
ring and $R\unlhd F$. Since $R/[F,R]$ is a central ideal of $F/[F,R]$,
the short exact sequence
\[
\begin{CD}
  0 @>>> R/[F,R] @>>> F/[F,R]     @>>>          L    @>>> 0 \\
\end{CD}
\]
is a central extension.
\begin{lem}\label{8}
 Suppose $0\to A\to B \xto\varphi C\to 0$ is a central extension and
$\alpha:L\to C$ is a homomorphism. Then there is a homomorphism $\beta:F/[F,R]\to B$
making the following diagram commutative:
\[
\begin{CD}
  0 @>>> R/[F,R] @>>> F/[F,R]     @>>>          L    @>>> 0 \\
  @.     @VVV         @VV{\beta}V               @VV{\alpha}V      @.\\
  0 @>>> A       @>>> B           @>{\varphi}>> C    @>>> 0 \\
\end{CD}
\]
Here the map from $R/[F,R]$ to $A$ is the restriction of $\beta$ to $R/[F,R]$.

\end{lem}
\begin{proof}
 As $F$ is free, there is a homomorphism $f:F\to B$ such that the following diagram commutes:
\[
\begin{CD}
  F       @>>> L            \\
  @VV{f}V      @VV{\alpha}V \\
  B       @>{\varphi}>> C            \\
\end{CD}
\]
Thus $f$ maps the Lie ring $R$ into $\ker\varphi=A$.
Now for $x\in F$ and $y\in R$ we have $f[x,y]=[f(x),f(y)]=0$
since $f(y)\in A$ and $A$ is a central ideal of $B$. Thus $[F,R]\subseteq \ker f$.
It follows that $f$ induces a homomorphism $\beta:F/[F,R]\to B$ such that for $\gamma=\beta_{|_{R/[F,R]}}$,
the diagram
\[
\begin{CD}
  0 @>>> R/[F,R]       @>>> F/[F,R]     @>>>          L            @>>> 0 \\
  @.     @VV{\gamma}V       @VV{\beta}V               @VV{\alpha}V      @.\\
  0 @>>> A             @>>> B           @>{\varphi}>> C            @>>> 0 \\
\end{CD}
\]
is commutative.
\end{proof}
Consider the set $\Hom(L,A)$ of all homomorphisms from $L$ into the abelian Lie ring $A$.
It is routine to check that this set forms an abelian group under pointwise addition.
The identity element of $\Hom(L,A)$ is the zero homomorphism and the inverse
$-\alpha$ of $\alpha$ is given by $(-\alpha)(x)=-\alpha(x)$ for all $x\in L$.

Suppose $\beta:B\to L$ is a Lie ring homomorphism. Then we can assign to any homomorphism $\alpha:L\to A$ the homomorphism
$\alpha \circ\beta:B\to A$. Thus we obtain a map $\beta_*:\Hom(L,A)\to \Hom(B,A)$ with
$\beta_*(\alpha)=\alpha \circ\beta$. We easily verify that $\beta_*$ is a homomorphism. Similarly,
suppose that $\beta:A\to B$ is a homomorphism of abelian Lie rings. Then each homomorphism
$\alpha:L\to A$ determines the homomorphism $\beta \circ \alpha:L\to A$ which defines a homomorphism
$\beta^*:\Hom(L,A)\to \Hom(L,B)$.

The above definitions lead to the following lemmas.
\begin{lem}\label{9}
Let $0\to L_1\xto\alpha L_2\xto\beta L_3\to 0$ be a short exact
sequence. Then the induced sequence
$$0\to \Hom(L_3,A)\xto{\beta_*} \Hom(L_2,A)\xto{\alpha_*}\Hom(L_1,A)$$
is exact.
\end{lem}
\begin{lem}\label{10}
 Let $0\to A\xto{\alpha}B\xto{\beta}C\to 0$ be a short exact sequence of
abelian Lie rings. Then, for an arbitrary Lie ring $L$, the induced sequence
$$0\to \Hom(L,A)\xto{\alpha^*}\Hom(L,B)\xto{\beta^*}\Hom(L,C)$$
is exact.
\end{lem}
Now we introduce two maps $\Inf$ and $\Tra$. Suppose that $\beta:L\to B$ is a Lie ring homomorphism.
We define the \Defn{inflation map} $\Inf$ by
  \[
  \begin{array}{lll}
  \Inf:& H^2(B,A)\to H^2(L,A) ,& \overline{(f,g)} \mapsto \overline{(f',g')}, 
  \end{array}
  \]
where
$f'(x,y)=f(\beta(x),\beta(y))$ and $g'(x,y)=g(\beta(x),\beta(y))$ for all $x,y\in L$.
\begin{lem}
 The map $\Inf$ is a well-defined injective homomorphism.
\end{lem}
\begin{proof}
We first show that $\Inf$ is well-defined.
Assume that $\overline{(f_1,g_1)}=\overline{(f_2,g_2)}\in H^2(B,A)$. Consequently there is a map $t:B\to A$ such that
$(f_1-f_2)(x,y)=-t[x,y]$ and $(g_1-g_2)(x,y)=t(x)+t(y)-t(x+y)$ for all $x,y \in L$. Since $\beta$ is a homomorphism,
we conclude that
  \begin{align*}
    (f_1'-f_2')(x,y)
    &= (f_1-f_2)(\beta(x),\beta(y))=-t[\beta(x),\beta(y)]=-(t\circ\beta)[x,y],\\
    (g_1'-g_2')(x,y)
    &= (g_1-g_2)(\beta(x),\beta(y))=t(\beta(x))+t(\beta(y))-t(\beta(x)+\beta(y))\\
    &= (t\circ\beta)(x)+(t\circ\beta)(y)-(t\circ\beta)(x+y).
  \end{align*}
It follows that $\overline{(f_1',g_1')}=\overline{(f_2',g_2')}$ since $t\circ\beta$ is a map from $L$ to $A$ with $(t\circ\beta)(0)=0$.
It is easily seen that $\Inf$ is a homomorphism.
\end{proof}

Next we define \Defn{transgression map} $\Tra$. Let $0\to H\xto{i}L\xto{\alpha}B\to 0$ be a central
Lie ring extension. Choose a section $\lambda$ of $\alpha$ and define the corresponding cocycle $(f,g)$ as in Section \ref{sec:centext-to-cohom}, that is,
  \[
  \begin{array}{lll}
  f :& B\times B \to H  ,& (x,y) \mapsto [\lambda(x),\lambda(y)] - \lambda[x,y], \\
  g :& B\times B \to H  ,& (x,y) \mapsto \lambda(x)+\lambda(y) - \lambda(x+y)
  .
  \end{array}
  \]
Given any $\chi\in \Hom(H,A)$, it is straightforward to verify that  $(\chi \circ f,\chi \circ g)\in Z^2(B,A)$.
Accordingly, we define the $\Tra$ by
  \[
  \begin{array}{lll}
  \Tra:& \Hom(H,A)\to H^2(B,A) ,& \chi\mapsto \overline{(\chi \circ f, \chi \circ g)},
  \end{array}
  \]
which one readily verifies to be a well-defined homomorphism.

\begin{thm}\label{thm:exact-sequence}
Let $H$ be a central ideal of $L$, and let $0\to H\xto{i}L\xto{\pi}L/H\to 0$ be the corresponding natural
exact sequence. Then the induced sequence
\[
0 \to \Hom(L/H,A)\xto{\Inf}\Hom(L,A)\xto{\Res}\Hom(H,A)\xto{\Tra}H^2(L/H,A)\xto{\Inf}H^2(L,A)
\]
 is exact, where $\Res$ denotes the restriction map.
\end{thm}
\begin{proof}
 By Lemma~\ref{9}, we need only verify the exactness at $\Hom(H,A)$ and $H^2(L/H,A)$. Let $\mu$
denote a section of $\pi$ such that $\mu(0)=0$ and let $(f,g)$ be the corresponding cocycle, that is, $f(x,y)=[\mu(x),\mu(y)]-\mu[x,y]$ and $g(x,y)=\mu(x)+\mu(y)-\mu(x+y)$.
\begin{enumerate}
\item
Exactness at $\Hom(H,A)$. We first observe that $\im(\Res)$ consists  of all homomorphisms $\chi:H\to A$
that can be extended to a homomorphism $\chi':L\to A$. Another preliminary observation is the following.
Fix $\chi \in \Hom(H,A)$ and consider the central extension $0\to A\xto{i}B\xto{\beta}L/H\to 0$
whose cohomology class is $\Tra(\chi)$. Then we may choose a section, say $\lambda$, of $\beta$ such that
$(\chi \circ f)(s_1,s_2)=[\lambda(s_1),\lambda(s_2)]-\lambda[s_1,s_2]$ and $(\chi \circ g)(s_1,s_2)=\lambda(s_1)+\lambda(s_2)-\lambda(s_1+s_2)$
for $s_1,s_2\in L/H$. Consequently, we claim that the map $\sigma:L\to B$ defined by $\sigma(h+\mu(s))=\chi(h)+\lambda(s)$
for $h\in H$ and $s\in L/H$ is a homomorphism.
 Clearly $\sigma\mid_{H}=\chi$. For the rest of the proof, put $l_1,l_2\in L$ in the form
$l_i=h_i+\mu(s_i)$ for $i=1,2$ such that $h_i\in H$ and $s_i\in L/H$. First we show that $\sigma$ is a bilinear. We have
  \begin{align*}
    \sigma(l_1+l_2)
    &= \sigma(h_1+\mu(s_1)+h_2+\mu(s_2))\\
    &= \sigma(h_1+h_2+g(s_1,s_2)+\mu(s_1+s_2)\\
    &= \chi(h_1+h_2+g(s_1,s_2))+\lambda(s_1+s_2)\\
    &= \chi(h_1)+\chi(h_2)+\chi(g(s_1,s_2))+\lambda(s_1)+\lambda(s_2)-\chi(g(s_1,s_2))\\
    &= \chi(h_1)+\lambda(s_1)+\chi(h_2)+\lambda(s_2)\\
    &= \sigma(h_1+\mu(s_1))+\sigma(h_1+\mu(s_1))=\sigma(l_1)+\sigma(l_2).
  \end{align*}
So $\sigma$ is bilinear. Next we prove that $\sigma[l_1,l_2]=[\sigma(l_1),\sigma(l_2)]$.
Since $h_i\in H$ and $H\subseteq Z(L)$, we can check that $\sigma[l_1,l_2]=(\chi \circ f)(s_1,s_2)+\lambda[s_1,s_2]$.
On the other hand, by $\chi(h_i)\in Z(B)$ we can write
  \begin{align*}
    [\sigma(l_1),\sigma(l_2)]
    &= [\sigma(h_1\mu(s_1),\sigma(h_2+\mu(s_2))]
     =[\chi(h_1)+\lambda(s_1),\chi(h_2)+\lambda(s_2)]\\
    &= [\lambda(s_1),\lambda(s_2)]=\chi of(s_1,s_2)+\lambda[s_1,s_2].
  \end{align*}
Therefore, $\sigma[l_1,l_2]=[\sigma(l_1),\sigma(l_2)]$. Hence $\sigma$ is a homomorphism.

Now assume that $\chi \in \ker(\Tra)$. Thus there is a homomorphism $\theta:L/H\to B$
with $\beta \circ \theta=1_{L/H}$. Then $B=A\oplus\theta(L/H)$ and the composite map $L\to A\oplus\theta(L/H)\to A$
provides a homomorphism that extends $\chi$. Hence $\ker(\Tra)\subseteq \im(\Res)$.

Conversely, let $\chi\in \im(\Res)$, so there is a $\chi'\in \Hom(L,A)$ such that $\chi'_{|H}=\chi$. Consider the map
$\psi:L\to B$ defined by $\psi(l)=-\chi'(l)+\sigma(l)$. Obviously $\psi$ is a homomorphism.
Now for $l\in L$, we have
  \begin{align*}
    \psi(l)
    &= \psi(h+\mu(s))=-\chi'(h+\mu(s))+\sigma(h+\mu(s))\\
    &= -\chi'(h)-\chi'(\mu(s))+\chi(h)+\lambda(s)=-\chi(h)-\chi'(\mu(s))+\chi(h)+\lambda(s)\\
    &= -\chi'(\mu(s)) +\lambda(s).
  \end{align*}
Since $\psi(l)=-\chi'(\mu(s)) +\lambda(s)$, then $\psi$ factors through $H$ and therefore determines a
homomorphism $\psi':L/H\to B$. If we define $\theta:L/H\to L/H$ by putting $\theta(s)=\mu(s)+H$
then $\beta \circ (\psi'\circ\theta)=1_{L/H}$. It follows that $\chi\in \ker(\Tra)$ so $\im(\Res)\subseteq \ker(\Tra)$.

\item
Exactness at $H^2(L/H,A)$. Assume that $\overline{(f_1,g_1)}\in \ker(\Inf)$, that is,
there is some $\gamma:L\to A$ with $\gamma(0)=0$ such that
$f_1(l_1+H,l_2+H)=-\gamma[l_1,l_2]$ and $g_1(l_1+H,l_2+H)=\gamma(l_1)+\gamma(l_2)-\gamma(l_1+l_2)$.
If $l_1\in H$ and $l_2\in H$ then $f_1(l_1+H,l_2+H)=0$ and $g_1(l_1+H,l_2+H)=0$ hence
$\gamma(l_1)+\gamma(l_2)=\gamma(l_1+l_2)$. Therefore the map $\chi:H\to A$ defined by
$\chi(h)=-\gamma(h)$ is a homomorphism. Set $f_2=\chi \circ f$ , $g_2=\chi \circ g$ and $t=\gamma \circ \mu$. By definition of
$f_1$ and $g_1$, we obtain $\gamma[\mu(s_1),\mu(s_2)]=-\lambda[s_1,s_2]+t[s_1,s_2]$ and
$\gamma(\mu(s_1 )+\mu(s_2))=-(\chi \circ g)(s_1,s_2)+t(s_1+s_2)$. Then we have
  \begin{align*}
    f_1(s_1,s_2)
    &= -\gamma[\mu(s_1),\mu(s_2)]=-\lambda[s_1,s_2]+t[s_1,s_2]\\
    &= -(\chi \circ f)(s_1,s_2)+t[s_1,s_2]\\
    g_1(s_1,s_2)
    &= \gamma(\mu(s_1))+\gamma(\mu(s_2))-\gamma(\mu(s_1)+\mu(s_2))\\
    &= \gamma(\mu(s_1))+\gamma(\mu(s_2))+(\chi \circ g)(s_1,s_2)-t(s_1+s_2)\\
    &= t(s_1)+t(s_2)-t(s_1+s_2)+(\chi \circ g)(s_1,s_2)
    .
  \end{align*}
Therefore, $\Tra(\chi)=\overline{(f_2,g_2)}=\overline{(f_1,g_1)}$, so $\ker(\Inf)\subseteq \im(\Tra)$.\\
Finally, any element of $\im(\Tra)$ has the form $\overline{(f_1,g_1)}$, where $f_1=\chi \circ f$ and
$g_1=\chi \circ g$ for some $\chi \in \Hom(H,A)$. Then $\Inf(\overline{(f_1,g_1)})=\overline{(f_2,g_2)}$ where
$f_2(l_1,l_2)=f_1(l_1+H,l_2+H)=f_1(s_1,s_2)$ and $g_2(l_1,l_2)=g_1(l_1+H,l_2+H)=g_1(s_1,s_2)$.
Next consider the mapping $\theta:L\to A$ defined by $\theta(h+\mu(s))=-\chi(h)$.
It is easy to show $[l_1,l_2]=f(s_1,s_2)+\mu[s_1,s_2]$. By definitions of $f_2$ and $g_2$, we deduce that
$f_2(l_1,l_2)=-\theta[l_1,l_2]$ and $g_2(l_1,l_2)=\theta(l_1)+\theta(l_2)-\theta(l_1+l_2)$.
Hence $\im(\Tra)\subseteq \ker(\Inf)$ as asserted.
\qedhere
\end{enumerate}
\end{proof}

\section{A Hopf formula for second cohomolgy}  \label{sec:schur-mult}
%  ___         _   _          
% / __| ___ __| |_(_)___ _ _  
% \__ \/ -_) _|  _| / _ \ ' \ 
% |___/\___\__|\__|_\___/_||_|

We now prove a version of Hopf's formula for the second cohomology of finite Lie ring.
The results and proofs found here unsurprisingly are very similar to
the corresponding ones in the theory of (finite) groups and also of Lie algebras.
\medskip

A Lie ring $L$ is said to be \Defn{divisible} if  for every positive integer $n$ and every $x\in L$, there exists an element $y\in L$ such that $ny=x$.

\begin{lem}\label{14}
 Suppose that $A$, $B$ and $C$ are abelian Lie rings such that $B\subseteq A$ and $C$ is divisible. Then any homomorphism
$B\to C$ extends to a homomorphism $A \to C$.
\end{lem}
\begin{proof}
The proof is similar to that of \cite[Lemma 2.1.6]{Karpilovsky:1987ta}.
 \end{proof}

The following is an adaption of a well-known statement for groups,
originally proven in \cite{Schur:1904dk}.
\begin{prop}\label{15}
Let $K$ be a central ideal of $L$ such that $L^2\cap K$ is finite. Then $L^2\cap K$ is isomorphic to the
image of the transgression map $\Tra:\Hom(K,\C^*)\to H^2(L/K,\C^*)$. In particular,
\begin{enumerate}
\item $L^2\cap K$ is isomorphic to a Lie subring of $H^2(L/K,\C^*)$,
\item $L^2\cap K\cong H^2(L/K,\C^*)$ if the map above is surjective.
\end{enumerate}
 \end{prop}
\begin{proof}
The proof is closely based on that of \cite[Lemma 2.1.7]{Karpilovsky:1987ta}.
Let $0\to K\to L\to L/K \to 0$ be the natural exact sequence. By Theorem~\ref{thm:exact-sequence},
the associated sequence
\[\Hom(L,\C^*)\xto{\Res}\Hom(K,\C^*)\xto{\Tra}H^2(L/K,\C^*)\]
is exact. Thus, if $H:=\im(\Res)$ denotes the  Lie subring of $\Hom(K,\C^*)$ consisting of all $\chi$ that can
be extended to a homomorphism $L\to \C^*$, then $H=\ker(\Tra)$ and so $\Hom(K,\C^*)/H\cong \im(\Tra)$. It remains to
prove that $L^2\cap K\cong \Hom(K,\C^*)/H$. Since $\Hom(L^2\cap K,\C^*)\cong L^2\cap K$, we only need to show that
$\Hom(K,\C^*)/H\cong \Hom(L^2\cap K,\C^*)$. By Lemma~\ref{14}, we conclude that the restriction map
$\psi:\Hom(K,\C^*)\to\Hom(L^2\cap K,\C^*)$ is surjective. It remains to show that $\ker\psi=H$.

Suppose now that $\chi \in H$. Then there is $\chi^*\in \Hom(L,\C^*)$ with $\chi^*\mid_{K}=\chi$.
As $\C^*$ is abelian, we have $L^2\subseteq \ker\chi^*$ and thus $L^2\cap K\subseteq \ker\chi$. 
Conversely, suppose $\chi\in\ker\psi$, thus $L^2\cap K\subseteq \ker\chi$.
By the isomorphism theorem, ${K}/(L^2\cap K) \cong {(K+L^2)}/{L^2}$ holds,
thus $\chi$ induces a homomorphism ${(K+L^2)}/{L^2}\to \C^*$. Due to
Lemma~\ref{14}, this can be extend to a homomorphism $L/L^2\to \C^*$ and hence to a homomorphism
$\chi^*:L\to \C^*$. One readily verifies that $\chi^*\mid_{K}=\chi$.
It follows that $H$ is the kernel of the homomorphism $\psi$.
\end{proof}

\begin{lem}\label{16}
 Let $L\cong {F}/{R}$, where $F$ is a free Lie ring. Then for any abelian Lie ring $A$, the transgression
map $\Hom({R}/{[F,R]},A)\to H^2(L,A)$ associated with the natural exact sequence
$0\to {R}/{[F,R]}\to{F}/{[F,R]}\xto{\varphi}L\to 0$ is surjective.
\end{lem}
\begin{proof}
 Fix $\overline{(f,g)}\in H^2(L,A)$ and a central extension $0\to A\to L^*\xto{\psi}L\to 0$
associated with $\overline{(f,g)}$. Then by Theorem~\ref{8}, there is a homomorphism $\theta:{F}/{[F,R]}\to L^*$
such that for $\gamma:= \theta\mid_{R/[F,R]}$ the diagram
\[
\begin{CD}
  0 @>>> R/[F,R]       @>>> F/[F,R]     @>{\varphi}>>          L            @>>> 0 \\
  @.     @VV{\gamma}V       @VV{\theta}V               @|      @.\\
  0 @>>> A             @>>> L^*           @>{\psi}>> L            @>>> 0 \\
\end{CD}
\]\\
commutes. Let $\mu$ be a section of $\varphi$. Then by commutativity
of the diagram, $\lambda:=\theta \circ \mu$ is
a section of $\psi$ and therefore the cocycle $(f_1,g_1)\in Z^2(L,A)$
defined by $f_1(x,y)=[\lambda(x),\lambda(y)]-\lambda[x,y]$ and
$g_1(x,y)=\lambda(x)+\lambda(y)-\lambda(x+y)$ for $x,y\in L$ is
cohomologous to $(f,g)$. It is clear that
$f_1(x,y)=\gamma([\mu(x),\mu(y)]-\mu[x,y])$ and
$g_1(x,y)=\gamma(\mu(x)+\mu(y)-\mu(x+y))$ for $x,y\in L$.
Thus, $\overline{(f,g)}$ is the image of $\gamma$ under the transgression map.
\end{proof}
\begin{lem}[{\cite[Lemma 7]{Eick/Horn/Zandi}}] \label{lem:L2-finite}
Let $L$ be a Lie ring. If $Z(L)$ has finite index in $L$, then $L^2$ is finite.
More precisely, if the additive group of $L/Z(L)$ has rank $n$ and exponent $e$, then $L^2$ has rank at most $n(n-1)/2$
and its exponent divides $e$.
\end{lem}

\begin{thm} \label{thm:H^2-from-FR}
 Let $L=F/R$ be a finite Lie ring, where $F$ is a free Lie ring. Then
$H^2(L,\C^*)=(F^2\cap R)/{[F,R]}$ and $H^2(L,\C^*)$ is finite.
\end{thm}
\begin{proof}
 Set $\overline{R}={R}/{[F,R]}$ and $\overline{F}={F}/{[F,R]}$. It is clear that $\overline{R}$ is a central ideal
of finite index in $\overline{F}$. By Lemma~\ref{lem:L2-finite}, it follows 
that $\overline{F}^2$ is finite, hence so is
$\overline{R}\cap \overline{F}^2$. Also, by Lemma~\ref{16} the transgression map
$\Hom(\overline{R},\C^*)\to H^2(L,\C^*)$ associated with the central extension
$0\to \overline{R}\to\overline{F}\to L\to 0$ is surjective.
Consequently, Proposition~\ref{15} (ii) implies that
$H^2(L,\C^*)\cong \overline{R}\cap \overline{F}^2$. As $[F,R]\subseteq F^2$, we have
$\overline{F}^2={F^2}/{[F,R]}$ whence as claimed
$\overline{R}\cap \overline{F}^2=(F^2\cap R)/{[F,R]}$.
\end{proof}

%------------------------------------------------------------------------------------%

\section*{Acknowledgments}
The authors are indebted to Professor A. Jamali  for his helpful suggestions,
to Bettina Eick for bringing them together and drawing their attention to this
topic in the first place, and also the referee for valuable comments that helped
to improve the exposition.

%------------------------------------------------------------------------------------%

%-----------------------------------------------------------------------------
%-----------------------------------------------------------------------------

\bigskip
\bigskip

{\footnotesize \par\noindent{Max Horn}\; \\
{Mathematisches Institut},
{Justus-Liebig-Universit{\"a}t Gie{\ss}en, }\\
{Arndtstrasse 2, D-35392 Gie{\ss}en, Germany}\\
Email: \url{max.horn@math.uni-giessen.de}\\

{\footnotesize \par\noindent{Seiran Zandi}\; \\
{Department of Mathematics}, {University of Kharazmi}, \\
{P.O.Box 15614,} {Tehran, Iran}\\
Email: \url{seiran.zandi@gmail.com}\\

\end{document}